\documentclass[a4paper,12pt,reqno]{amsart}

\usepackage{amsmath}
\usepackage{bbm}
\usepackage{amssymb}
\usepackage{amsfonts}
\usepackage{graphicx}
\usepackage{mathtools}
\usepackage{xcolor}
\usepackage[colorlinks]{hyperref}
\renewcommand\eqref[1]{(\ref{#1})}

\graphicspath{ {images/} }
\title[Multidimensional Hardy-FLW Inequality]{Multidimensional Frank-Laptev-Weidl improvement of the Hardy Inequality}

\author[P. Roychowdhury]{Prasun Roychowdhury}
\address{
	Prasun Roychowdhury:
	\endgraf
	Department of Mathematics: Analysis, Logic and Discrete Mathematics
	\endgraf
	Ghent University, Krijgslaan 281, Building S8, B 9000 Ghent
	\endgraf
	Belgium
	\endgraf
	{\it E-mail address} {\rm prasunroychowdhury1994@gmail.com}}

\author[M. Ruzhansky]{Michael Ruzhansky}
\address{
	Michael Ruzhansky:
	\endgraf
	Department of Mathematics: Analysis, Logic and Discrete Mathematics
	\endgraf
	Ghent University, Krijgslaan 281, Building S8, B 9000 Ghent
	\endgraf
	Belgium
	\endgraf
	and
	\endgraf
	School of Mathematical Sciences
	\endgraf Queen  Mary University of London 
	\endgraf
	United Kingdom
	\endgraf
	{\it E-mail address} {\rm michael.ruzhansky@ugent.be}}

\author[D. Suragan]{Durvudkhan Suragan}
\address{
	Durvudkhan Suragan:
	\endgraf
	Department of Mathematics
	\endgraf
	Nazarbayev University
	\endgraf
	Kazakhstan
	\endgraf
	{\it E-mail address} {\rm durvudkhan.suragan@nu.edu.kz}}

\date{\today}

\theoremstyle{plain}

\newtheorem{theorem}{Theorem}[section]

\newtheorem{lemma}{Lemma}[section]
\newtheorem{corollary}{Corollary}[section]
\newtheorem{remark}{Remark}[section]

\numberwithin{equation}{section} \allowdisplaybreaks

\usepackage[text={6in,8.6in},centering]{geometry}
\newcommand{\g}{\mathbb{R}^N}
\newcommand{\rn}{\mathbb{R}^N}

\newcommand{\sn}{\mathbb{S}^{N-1}}

\newcommand{\dr}{\:{\rm d}r}
\newcommand{\dx}{\:{\rm d}x}

\newcommand{\dt}{\:{\rm d}t}

\newcommand{\dsn}{\:{\rm d}\sigma}

\begin{document}
	\begin{abstract}
		We establish a new improvement of the classical $L^p$-Hardy inequality on the multidimensional Euclidean space in the supercritical case. Recently, in \cite{rupert-22}, there has been a new kind of development of the one-dimensional Hardy inequality. Using some radialisation techniques of functions and then exploiting symmetric decreasing rearrangement arguments on the real line, the new multidimensional version of the Hardy inequality is given. Some consequences are also discussed.
	\end{abstract}
	\maketitle
	
	\section{Introduction}
	
	Nowadays, one of the most popular classical functional inequalities in the analysis is the Hardy inequality 
	\begin{equation}\label{Hardy-Leray}
		\int_{\mathbb{R}^N} \frac{|u(x)|^{p}}{|x|^{p}}\dx\leq \left|\frac{p}{N-p}\right|^{p}
		\int_{\mathbb{R}^{N}}|\nabla u(x)|^{p}\dx, \quad N<p<\infty,
	\end{equation}
	which holds for all $u\in C_{c}^\infty(\rn \setminus \{o\})$. This range of $p$ is called the supercritical case of the Hardy inequality in literature.  Note that
	the inequality is also valid for all $u\in C_{c}^\infty(\rn)$ if $N>p$, that is, in the subcritical case. However, in the critical case $N=p$ such inequality is not possible. 
	
	The inequality \eqref{Hardy-Leray} is an essential 
	higher dimensional extension of the one-dimen- sional inequality discovered by Hardy \cite{hardy}. The development of the famous Hardy inequality (in both its discrete and continuous forms) during the period 1906-1928 has its own history and we refer to \cite{prehistory}. It is well-known that the so-called Hardy constant
	$\left|\frac{p}{N-p}\right|^{p}$ is sharp and never attained (except trivial function).

	Therefore, one may want to improve \eqref{Hardy-Leray} by adding extra nonnegative terms on its left-hand side. Say $p=2$ and $N>2$, one may ask about the existence of nonnegative function $W\in L^1(\rn)$ such that the following inequality \begin{equation*}
		\int_{\mathbb{R}^N} W(x)|u|^{2}\dx+\int_{\mathbb{R}^N} \frac{|u(x)|^{2}}{|x|^{2}}\dx \leq 
		\left(\frac{2}{N-2}\right)^{2}\int_{\mathbb{R}^{N}}|\nabla u(x)|^{2}\dx,
	\end{equation*}
	holds for all $u\in C_c^\infty(\rn)$. But the operator $-\Delta_{\rn}-\frac{(N-2)^2}{4}\frac{1}{|x|^2}$ is known to be a \emph{critical} operator on $\rn\setminus\{o\}$ (see \cite{critical}) and an improvement of such quadratic form inequality is not possible. Also see \cite{dp} for the optimal $L^p$ Hardy-type inequalities. However, there is a huge set of references of works about improved Hardy inequalities on \emph{bounded} Euclidean domains after the seminal works of Brezis and Marcus \cite{brezis-1} and Brezis and V\'azquez \cite{brezis-2}. See also \cite{adi,adi-2,vaz} and references therein.

	The Hardy inequality \eqref{Hardy-Leray} plays an important role in several branches of mathematics such as partial differential equations, spectral theory, geometry, functional analysis, etc. Improvements of the Hardy inequality on bounded Euclidean domains containing origin and improvements of such inequality on Riemannian manifolds have attracted great attention and were investigated by many authors. Without any claim of completeness, we refer an interested reader to \cite{BFT,BFT2,FT,gaz,GM,hardy-book-rs, tak} which are excellent monographs for reviews of this subject and for the improvements of this inequality.

	Let $p > N $. In this paper, our main result states that for all $u\in C_{c}^\infty(\rn \setminus \{o\})$, the following new sharp inequality holds in terms of the \emph{polar coordinate} structure of $\rn$:
	\begin{align}\label{improvedHardy}
		\max \biggl\{\int_{0}^\infty r^{N-1} \sup_{0<s\leq r}\int_{\sn}&\frac{|u(s\sigma)|^p}{r^p}\dsn\dr,  \int_{0}^\infty  r^{N-1}\sup_{r\leq s<\infty}\int_{\sn}\frac{|u(s\sigma)|^p}{s^p}\dsn\dr\biggr\}  \nonumber\\&\leq   \bigg|\frac{p}{N-p}\bigg|^p \int_{\rn} \left|\nabla u(x)\right|^p{\rm d}x.
	\end{align}
	Here $r$ is the distance between a point $x\in\rn$ and the origin $o$ and $\sn$ is the $N$-dimensional unit sphere. Clearly, we have (see Remark \ref{rem-5})
	\begin{align*}
		\max \biggl\{\int_{0}^\infty r^{N-1} \sup_{0<s\leq r}\int_{\sn}&\frac{|u(s\sigma)|^p}{r^p}\dsn\dr,  \int_{0}^\infty  r^{N-1}\sup_{r\leq s<\infty}\int_{\sn}\frac{|u(s\sigma)|^p}{s^p}\dsn\dr\biggr\}  \\&\geq   \int_{\mathbb{R}^N} \frac{|u(x)|^{p}}{|x|^{p}}\dx,
	\end{align*}
	so that \eqref{improvedHardy} gives an improvement of \eqref{Hardy-Leray}. The $N=1$ case of \eqref{improvedHardy} (see also Theorem \ref{main_th_rad}) was established in the recent paper \cite{rupert-22}, that is, the authors proved the one dimensional $L^{p}$-version of the improved Hardy inequality and gave an interesting application in the theory of Schr\"odinger operators. So, our inequality \eqref{improvedHardy} extends the 1D Frank-Laptev-Weidl inequality from \cite{rupert-22} to dimension $1\leq N<p$. The works on the one-dimensional (similar) improvements of the Hardy inequality go back to \cite{Kac-Krein}, and \cite{Tomaselli}, see also the introductory discussions in \cite{rupert-22} and \cite{rs} for the discrete versions.

	First, we study the results for radial functions and then continue the discussion for the non-radial setup. One of the main tools we exploited is the norm-preservation of the symmetric decreasing rearrangements and in principle, one can see that this property holds on some Riemannian or/and sub-Riemannian manifolds. In the same spirit, our results can be extended to more general manifolds/spaces. Here are few references \cite{EGR-21,Carron,serena,FLLM,Kombe,VHN,SurveyHomGroupHardy,YSK} to revisit the work on Hardy's inequality in those spaces.

	Structure of the paper: In Section \ref{2}, we discuss some basic facts of symmetric decreasing rearrangements. Section \ref{3} is devoted to the main supporting lemmas, and then a few necessary tools are discussed. In Section \ref{4}, we prove our main results related to developing the new multidimensional Hardy inequality. Finally, in Section \ref{5}, a novel uncertainty principle on the Euclidean space is discussed.

	\section{Preliminaries}\label{2}
	Before stating the main results and their consequences, first, we will describe some preliminaries on symmetric decreasing rearrangements. After that, in this section, we will shortly discuss the polar coordinate decomposition and the radial version of the classical gradient operator on the $N$-dimensional Euclidean space $\rn$.

	\subsection{Symmetric decreasing rearrangements} Below we will quickly recall some definitions and facts about \emph{symmetric decreasing rearrangement}. For more details, we refer to \cite[Chapter 3]{lieb}, for example.

	Let $\Omega\subset\rn$ be a finite Borel measurable subset. Then the symmetrization of $\Omega$ (denoted by $\Omega^*$) is defined by the open ball $B(o\:;\:r):=\{x\in\rn\::\:|x|<r\}$, where $r=v_N^{-1/N}\text{vol}(\Omega)^{1/N}$ is the radius, where $v_N$ is the volume of the unit $N$-dimensional Euclidean ball and $o$ is the origin as the centre of the ball. Here we are only concerned with functions which \emph{vanish at infinity}. For a real number $\beta\in \mathbb{R}$, the level set $\{f>\beta\}$ of a function $f$ is denoted as 
	\begin{align*}
		\{f>\beta\}:=\{x\in \rn\: :\: f(x)>\beta\}.	
	\end{align*}

	We say that a function $f$ vanishes at infinity if
	\begin{align*}
		\text{vol}(\{|f|>\tau\})<\infty \space \text{ for all }\tau>0.
	\end{align*}

	Now for all $x\in \rn$, we define the symmetric decreasing rearrangement (or non-increasing rearrangement) of $f$, denoted by $f^*$, as follows
	\begin{align*}
		f^*(x):=\int_0^\infty \chi_{\{|f|>\tau\}^*}(x)\:{\rm d}\tau.
	\end{align*}

	Then by definition, $f^*$ becomes a nonnegative, radially symmetric, and non-increa\\-sing function. Therefore, for any $x\in \rn$, we have $f^*(x)=f^*(|x|)$ and one can consider $f^*$ as a real valued nonnegative function on $[0,\infty)$. Irrespective of several properties of $f^*$, a useful property in our context is
	\begin{align*}
		\text{vol}(\{|f|>\tau\})=\text{vol}(\{f^*>\tau\}) \text{ for all }\tau\geq 0.	
	\end{align*}

	By using the \emph{layer cake representation} and the above property, we have the following identity:
	\begin{align}\label{layer}
		\int_{\mathbb{R}^N}|f(x)|^p\dx=\int_{\mathbb{R}^N}|f^*(x)|^p\dx\:\: \text{ for all }p\geq1.
	\end{align}
	This relation will be very useful in the proofs.

	\subsection{Polar coordinates and radial gradient}
	Let $\rn$ be the $N$-dimensional Euclidean space with Lebesgue measure ${\rm d}x$. Then it admits the \emph{polar coordinate decomposition} with respect to the origin $o\in \rn$. In particular, for any $f\in L^1_{loc}(\rn)$ we have
	\begin{align}\label{metric_polar}
		\int_{\rn} f(x) \dx=\int_{0}^{\infty}\int_{\sn}f(r,\sigma)\:r^{N-1}\dsn\dr,
	\end{align}
	where for any $x\in \rn$ we write $x=(r,\sigma)\in [0,\infty)\times \sn$ with $r=\varrho(x,o)$ (also denoted as $|x|$) being the Euclidean distance between $x$ and the fixed point $o$ as the origin. Here and after $\sn=\{x\in\rn\: : \:|x|=1\}$ is the $N$-dimensional unit sphere with the surface measure $\dsn$. 
	
	We say that a function is radially symmetric if it depends only on the radial part. That is, if $f(x)$ is a radial function, then for any $x\in\rn$, we have
	\begin{align*}
		f(x)=f(|x|)=f(r) \text{ where } r=|x|,
	\end{align*}
	and $f$ can be considered as a function on $[0,\infty)$. Note that the radial gradient of a differentiable function $f(x)$ can be defined by
	\begin{align}\label{iden}
		\frac{\partial f}{\partial r}(x)=\frac{x}{|x|}\cdot \nabla f(x),
	\end{align}
	where $``\cdot "$ is the scalar product and $\nabla$ is the usual gradient on $\rn$.
	\medspace
	
	\section{Supporting lemmas}\label{3}
	This section deals with the establishment of some supporting results. First, we describe the weighted version of the Hardy inequality on the half-line involving symmetric decreasing (or non-increasing) rearrangement of the function. The arguments here follow  \cite{rupert-22} but extend those.
	
	\begin{lemma}\label{key_lem}
		Let $1<p<\infty$. Let $g$ be any nonnegative function on $(0,\infty)$. Assume $h$ is a strictly positive non-decreasing function on $(0,\infty)$ such that  $s h(r)\leq r h(s)$ for any $r,s\in(0,\infty)$ with $r\leq s$. Let $f$ be a locally absolutely continuous function on $(0,\infty)$. Then we have
		\begin{equation}\label{eqn_key_lem}
			\int_{0}^{\infty}g(r)\sup_{0<s<\infty}\bigg|\min\biggl\{\frac{1}{h(r)},\frac{1}{h(s)}\biggr\}\int_{0}^{s}f(t)\dt\bigg|^p\dr \leq \int_{0}^{\infty}g(r)\bigg|\frac{1}{h(r)}\int_{0}^{r}f^*(t)\dt\bigg|^p\dr,
		\end{equation}
		where $f^*$ is the non-increasing rearrangement of $f$.
	\end{lemma}
	\begin{proof}
		For any fixed $r>0$, we have
		\begin{align*}
			\sup_{0<s<\infty}\bigg|\min\biggl\{\frac{1}{h(r)},\frac{1}{h(s)}\biggr\}\int_{0}^{s}f(t)\dt\bigg|\leq \sup_{0<s<\infty}\min\biggl\{\frac{1}{h(r)},\frac{1}{h(s)}\biggr\}\int_{0}^{s}f^*(t)\dt.
		\end{align*}
		The above follows from 
		\begin{align*}
			\bigg|\int_{0}^{s}f(t)\dt\bigg|\leq \int_{0}^{s}|f(t)|\dt \leq \int_{0}^{s}f^*(t)\dt.
		\end{align*}
		The advantage of $f^*$ is that it is non-increasing and this fact enables computing the supremum.
		
		{\bf Case 1:} Let $0<s\leq r<\infty$. Then using the non-decreasing property of $h$, we have 
		\begin{align*}
			\min\biggl\{\frac{1}{h(r)},\frac{1}{h(s)}\biggr\}\int_{0}^{s}f^*(t)\dt=\frac{1}{h(r)}\int_{0}^{s}f^*(t)\dt\leq \frac{1}{h(r)}\int_{0}^{r}f^*(t)\dt.
		\end{align*}
		
		{\bf Case 2:} Let $0<r\leq s<\infty$. Then exploiting the increasing property of $h$ and the non-increasing nature of $f^*$ and a change of variable, we obtain
		\begin{align*}
			&\min\biggl\{\frac{1}{h(r)},\frac{1}{h(s)}\biggr\}\int_{0}^{s}f^*(t)\dt\\&=\frac{1}{h(s)}\int_{0}^{s}f^*(t)\dt\leq \frac{1}{h(s)}\int_{0}^{s}f^*(rt/s)\dt\\&=\frac{s}{rh(s)}\int_{0}^{r}f^*(v)\:{\rm d}v\leq\frac{1}{h(r)}\int_{0}^{r}f^*(t)\:{\rm d}t.
		\end{align*}

		Thus, combining both cases, for all $r,s\in (0,\infty)$ we have 
		\begin{align*}
		\min\biggl\{\frac{1}{h(r)},\frac{1}{h(s)}\biggr\}\int_{0}^{s}f^*(t)\dt \leq \frac{1}{h(r)}\int_{0}^{r}f^*(t)\dt.
		\end{align*}
It yields  		
			\begin{align*}
			\sup_{0<s<\infty}\left|\min\left\{\frac{1}{h(r)},\frac{1}{h(s)}\right\}\int_{0}^{s}f(t)\dt \right|^{p} \leq \left| \frac{1}{h(r)}\int_{0}^{r}f^*(t)\dt\right|^{p},
		\end{align*}
		which completes the proof. 
	\end{proof}
	
	\medspace
	
	Let us consider a special case in Lemma \ref{key_lem}.
	\begin{corollary}\label{key_cor}
		Let $1\leq N<p< \infty$. Then for all $f\in C_c^\infty(0,\infty)$ the following weighted inequality holds:
		\begin{equation}\label{eqn_key_cor}
			\int_{0}^{\infty}r^{N-1}\sup_{0<s<\infty}\bigg|\min\biggl\{\frac{1}{r},\frac{1}{s}\biggr\}\int_{0}^{s}f(t)\dt\bigg|^p\dr \leq \bigg|\frac{p}{N-p}\bigg|^p\int_{0}^{\infty}r^{N-1}\:|f^*(r)|^p\dr,
		\end{equation}
		where $f^*$ is the non-increasing rearrangement of $f$. 
	\end{corollary}
	\begin{proof}
		Let us set $g(r)=r^{N-1}$ and $h(r)=r$ for $r\in (0,\infty)$. Substituting these in Lemma \ref{key_lem}, we have
		\begin{equation}\label{inters}
			\int_{0}^{\infty}r^{N-1}\sup_{0<s<\infty}\bigg|\min\biggl\{\frac{1}{r},\frac{1}{s}\biggr\}\int_{0}^{s}f(t)\dt\bigg|^p\dr \leq\int_{0}^{\infty}r^{N-p-1}\bigg|\int_{0}^{r}f^*(t)\dt\bigg|^p\dr.
		\end{equation}
		Using the weighted one dimensional $L^p$-Hardy inequality (see \cite[Theorem 330]{hlp}, or e.g. \cite[Theorem 3.1]{rst}) for the function $\int_{0}^{r}f^*(t)\dt$ and noticing the fact $\frac{\partial}{\partial r}\int_{0}^{r}f^*(t)\dt=f^*(r)$ which follows from the fundamental theorem of calculus, we obtain
		\begin{align*}
			\int_{0}^{\infty}r^{N-p-1}\bigg|\int_{0}^{r}f^*(t)\dt\bigg|^p\dr \leq \bigg|\frac{p}{p-N}\bigg|^p\int_{0}^{\infty}r^{N-1}|f^*(r)|^p\dr. 
		\end{align*}
		Applying this to the right-hand side of \eqref{inters}, we obtain \eqref{eqn_key_cor}.
	\end{proof}

	For a continuous function on a compact set, the supremum is attained, and exploiting this idea, one can have the following result.
	\begin{lemma}\label{help_lem_3}
		Let $g\in C(\rn)$ be a nonnegative function. Then, on a compact subset $\mathcal{K}\subset\rn$, we have
		\begin{align*}
			\big(\sup_{x\in\mathcal{K}}g(x)\big)^p=\big(\sup_{x\in\mathcal{K}}g^p(x)\big)
		\end{align*}
		for $1\leq p<\infty$.
	\end{lemma}

	\section{Improvement of the classical Hardy inequality}\label{4}
	This section's primary goal is to establish an improved version of the classical Hardy inequality on the $N$-dimensional Euclidean space $\rn$ in the supercritical case. Our strategy is first to develop it for radial functions, and then, by using the radialisation technique, we settle the non-radial version.

	\subsection{Radial version of the results}
	First, we present the results for the compactly supported smooth radial function space denoted as $C_{c,rad}^\infty(\rn\setminus \{o\})$.
	
	\begin{theorem}\label{main_th_rad}
		Let $1\leq N< p<\infty$. Then we have
		\begin{align}\label{eqn_main_th_rad}
			\int_{\rn} \max\biggl\{ \sup_{\bar{B}(o\:;\:|x|)\setminus\{o\}}\frac{|u(y)|^p}{|x|^p}\:,\:\sup_{B^c(o\:;\:|x|)}\frac{|u(y)|^p}{|y|^p}\biggr\}\dx\leq\bigg|\frac{p}{N-p}\bigg|^p\int_{\rn}\bigg|\frac{x}{|x|}\cdot \nabla u(x)\bigg|^p\dx
		\end{align}
		for all $u\in  C_{c,rad}^\infty(\rn\setminus \{o\})$.
	\end{theorem}
	\begin{proof}
		Since $u\in C_{c,rad}^\infty(\rn\setminus \{o\})$ we have $u(y)=u(|y|)=u(s)$ for $s=|y|$, that is, $u\in C_c^\infty(0,\infty)$. Recall the polar coordinate decomposition $x=(r,\sigma)$ where $r=|x|\in (0,\infty)$ and $\sigma=\frac{x}{|x|}\in \sn$. Then we deduce
		\begin{align}\label{im-1}
			&\int_{\rn} \max\biggl\{ \sup_{\bar{B}(o\:;\:|x|)\setminus\{o\}}\frac{|u(y)|^p}{|x|^p}\:,\: \sup_{B^c(o\:;\:|x|)}\frac{|u(y)|^p}{|y|^p}\biggr\}\dx\nonumber\\
			&= \int_{0}^\infty \int_{\sn}r^{N-1}\:\max\biggl\{ \sup_{0<s\leq r}\frac{|u(s)|^p}{r^p}\:,\: \sup_{r\leq s<\infty}\frac{|u(s)|^p}{s^p}\biggr\}\dsn\:\dr.  
		\end{align}
		Before going further let us mention the following identity
		\begin{align*}
			&\sup_{0<s<\infty}\bigg|\min\biggl\{\frac{1}{r},\frac{1}{s}\biggr\}\;u(s)\bigg|^p\nonumber\\&=\sup_{0<s<\infty}\min\biggl\{\frac{1}{r^p},\frac{1}{s^p}\biggr\}\;|u(s)|^p\nonumber\\&=\max\biggl\{ \sup_{0<s\leq r}\min\biggl\{\frac{1}{r^p},\frac{1}{s^p}\biggr\}\;|u(s)|^p\:,\: \sup_{r\leq s<\infty}\min\biggl\{\frac{1}{r^p},\frac{1}{s^p}\biggr\}\;|u(s)|^p\biggr\}\nonumber\\&=\max\biggl\{ \sup_{0<s\leq r}\frac{|u(s)|^p}{r^p}\:,\: \sup_{r\leq s<\infty}\frac{|u(s)|^p}{s^p}\biggr\}.
		\end{align*}
		By using this and continuing with the polar coordinate decomposition \eqref{im-1}, we have
		\begin{align*}
			&\int_{\rn} \max\biggl\{ \sup_{\bar{B}(o\:;\:|x|)\setminus\{o\}}\frac{|u(y)|^p}{|x|^p}\:,\: \sup_{B^c(o\:;\:|x|)}\frac{|u(y)|^p}{|y|^p}\biggr\}\dx\\
			&=\int_{\sn}\int_{0}^\infty r^{N-1}\:\sup_{0<s<\infty}\bigg|\min\biggl\{\frac{1}{r},\frac{1}{s}\biggr\}\;u(s)\bigg|^p\dr\dsn\\
			&=\int_{\sn}\int_{0}^\infty r^{N-1}\:\sup_{0<s<\infty}\bigg|\min\biggl\{\frac{1}{r},\frac{1}{s}\biggr\}\;\int_0^s\frac{\partial u}{\partial t}(t)\dt\bigg|^p\dr\dsn\\&\overset{\mathrm{ Corollary}\:\ref{key_cor} }{\leq}\bigg|\frac{p}{N-p}\bigg|^p \int_{\sn}\int_{0}^{\infty}r^{N-1}\bigg|\bigg(\frac{\partial u}{\partial r}\bigg)^*(r)\bigg|^p\dr\dsn\\
			&=\bigg|\frac{p}{N-p}\bigg|^p\int_{\rn}\bigg|\bigg(\frac{\partial u}{\partial |x|}\bigg)^*(x)\bigg|^p\dx\\
			&=\bigg|\frac{p}{N-p}\bigg|^p\int_{\rn}\bigg|\frac{\partial u}{\partial |x|}(x)\bigg|^p\dx\\
			&=\bigg|\frac{p}{N-p}\bigg|^p\int_{\rn} \bigg|\frac{x}{|x|}\cdot \nabla u(x)\bigg|^p\dx.
		\end{align*}
		In the middle we have used Corollary \ref{key_cor} for  $f(t)=\frac{\partial u}{\partial t}(t)$. Finally, by using \eqref{layer} for the function $\frac{\partial u}{\partial r}$ which vanishes at infinity because of the compact support of the smooth function $u$ and by using the identity \eqref{iden} the desired result follows.
	\end{proof}

	\subsection{Non-radial setting of the results}
	Now we describe the non-radial version of Theorem \ref{main_th_rad}. Before that, it should be mentioned that constructing some non-radial inequality from the radial one, the radialisation method is one of the common tools of functional inequalities. Let $u\in L^1(\rn)$, then for any $1<p<\infty$, we define the radial symmetric function $\tilde{u}$ as follows:
	\begin{align*}
		\tilde{u}(x)=\tilde{u}(r):=\bigg(\frac{1}{\omega_N}\int_{\sn}|u(r\sigma)|^p\dsn\bigg)^{\frac{1}{p}} \text{ for any }x\in\rn,
	\end{align*}
	where $r=|x|$, $\sigma=\frac{x}{|x|}$, and $\omega_N$ is the surface area of the $N$-dimensional sphere $\sn$.

	\begin{lemma}\label{tran_lem}
		Let $1<p<\infty$ and let $f$ be any nonnegative measurable radial function on $\rn$. Then for any $u\in C^1(\rn)$, we have
		\begin{equation}\label{eqn_tran_lem}
			\int_{\rn}f(x)\bigg|\frac{x}{|x|}\cdot \nabla \tilde{u}(x)\bigg|^p\dx\leq \int_{\rn}f(x)\bigg|\frac{x}{|x|}\cdot \nabla u(x)\bigg|^p\dx,
		\end{equation}
		where $\tilde{u}(x)$ is the earlier defined radial symmetric version of $u(x)$.
	\end{lemma}
	\begin{proof}
		Let $u\in C^1(\rn)$. By using the H\"older inequality, we have
		\begin{align*}
			&\bigg|\frac{x}{|x|}\cdot \nabla \tilde{u}(x)\bigg|	= \bigg|\frac{x}{|x|}\cdot \nabla \bigg(\frac{1}{\omega_N}\int_{\sn}|u(r\sigma)|^p\dsn\bigg)^{\frac{1}{p}}\bigg|\\&=\frac{1}{(\omega_N)^{1/p}}\:\frac{1}{p}\: \bigg(\int_{\sn}|u(r\sigma)|^p\dsn\bigg)^{\frac{1}{p}-1}\bigg|\int_{\sn}p\:|u(r\sigma)|^{p-2}u(r\sigma)\:\frac{x}{|x|}\cdot \nabla u(r\sigma)\dsn\bigg|\\&\leq \frac{1}{(\omega_N)^{1/p}}\: \bigg(\int_{\sn}|u(r\sigma)|^p\dsn\bigg)^{\frac{1-p}{p}}\times \\&\quad\quad\quad\quad\quad\quad\quad\quad\quad\quad\bigg(\int_{\sn}|u(r\sigma)|^p\dsn\bigg)^{\frac{p-1}{p}}\bigg(\int_{\sn}\bigg|\frac{x}{|x|}\cdot \nabla u(r\sigma)\bigg|^p\dsn\bigg)^{\frac{1}{p}}\\&= \frac{1}{(\omega_N)^{1/p}}\:\bigg(\int_{\sn}\bigg|\frac{x}{|x|}\cdot \nabla u(x)\bigg|^p\dsn\bigg)^{\frac{1}{p}}.
		\end{align*}
		
		Now multiplying both sides by $f(x)$ and exploiting the above we derive
		\begin{align*}
			&\int_{\rn}f(x)\bigg|\frac{x}{|x|}\cdot \nabla \tilde{u}(x)\bigg|^p\dx\\&=\int_{0}^\infty\int_{\sn}f(r)\:r^{N-1}\bigg|\frac{x}{|x|}\cdot \nabla \tilde{u}(x)\bigg|^p\dsn\dr\\&\leq\omega_N\int_{0}^\infty f(r)\:r^{N-1}\:\frac{1}{\omega_N}\:\bigg(\int_{\sn}\bigg|\frac{x}{|x|}\cdot \nabla u(x)\bigg|^p\dsn\bigg)\dr\\&=\int_{0}^\infty f(r)\:r^{N-1}\:\bigg(\int_{\sn}\bigg|\frac{x}{|x|}\cdot \nabla u(x)\bigg|^p\dsn\bigg)\dr\\&=\int_{\rn}f(x)\bigg|\frac{x}{|x|}\cdot \nabla u(x)\bigg|^p\dx.
		\end{align*}
		Starting with the polar coordinate decomposition, in between, we have used Fubini's theorem to arrive at the inequality \eqref{eqn_tran_lem}.
	\end{proof}

	Now we proceed with another important lemma.
	\begin{lemma}\label{sup_lem}
		Let  $1\leq p<\infty$ and let $f$ be a nonnegative measurable radial weight function on $\rn$. Then for any $u\in C_c(\rn\setminus\{o\})$, we have
		\begin{align}\label{eqn_sup_lem}
		\max \biggl\{\int_{0}^\infty f(r)\: &r^{N-1} \sup_{0<s\leq r}\int_{\sn}\frac{|u(s\sigma)|^p}{r^p}\dsn\dr,\nonumber\\&  \int_{0}^\infty  f(r)\:r^{N-1}\sup_{r\leq s<\infty}\int_{\sn}\frac{|u(s\sigma)|^p}{s^p}\dsn\dr\biggr\}\nonumber\\&\leq  \int_{\rn} f(x)\max\biggl\{ \sup_{\bar{B}(o\:;\:|x|)\setminus\{o\}}\frac{|\tilde{u}(y)|^p}{|x|^p}\:,\: \sup_{B^c(o\:;\:|x|)}\frac{|\tilde{u}(y)|^p}{|y|^p}\biggr\}\dx,
		\end{align}
		where $\tilde{u}(x)$ is the earlier defined radial symmetric version of $u(x)$.
	\end{lemma}
	\begin{proof}
		Setting $s=|y|$ and $r=|x|$ and using polar coordinates, we compute 
		\begin{align*}
			&\int_{\rn} f(x)\max\biggl\{ \sup_{\bar{B}(o\:;\:|x|)\setminus\{o\}}\frac{|\tilde{u}(y)|^p}{|x|^p}\:,\: \sup_{B^c(o\:;\:|x|)}\frac{|\tilde{u}(y)|^p}{|y|^p}\biggr\}\dx\\
			&=\int_{0}^\infty f(r)\:r^{N-1} \max\biggl\{ \sup_{\bar{B}(o;|x|)\setminus\{o\}}\frac{\int_{\sn}|u(s\sigma)|^p{\rm d}\sigma}{|x|^p}, \sup_{B^c(o;|x|)}\frac{\int_{\sn}|u(s\sigma)|^p{\rm d}\sigma}{|y|^p}\biggr\}\:{\rm d}r\\
			&=\int_{0}^\infty \max\biggl\{ f(r)\:r^{N-1} \sup_{0<s\leq r}\frac{\int_{\sn}|u(s\sigma)|^p\dsn}{r^p}, \\
			&\quad\quad\quad\quad\quad\quad\quad\quad f(r)\:r^{N-1} \sup_{r\leq s<\infty}\frac{\int_{\sn}|u(s\sigma)|^p\dsn}{s^p}\biggr\}\dr\\
			&\geq\max\biggl\{  \int_{0}^\infty f(r)r^{N-1} \sup_{0<s\leq r}\int_{\sn}\frac{|u(s\sigma)|^p}{r^p}\dsn\dr,\\& \quad\quad\quad\quad\quad\quad\quad\quad\int_{0}^\infty  f(r)r^{N-1}\sup_{r\leq s<\infty}\int_{\sn}\frac{|u(s\sigma)|^p}{s^p}\dsn\dr\biggr\}.
		\end{align*}
	Here we have used the fact $\int\max\{a(x),b(x)\}\,\dx\geq \max\{\int a(x)\,\dx,\int b(x)\,\dx\}$.
	\end{proof}

	We are now in a position to derive the non-radial version of the result from the previous subsection which is the main contribution of this note.
	\begin{theorem}\label{main_th}
		Let $1\leq N<p<\infty$. Then, for all $u\in C_{c}^\infty(\rn \setminus \{o\})$, we have
		\begin{align}\label{eqn_main_th}
				\max \biggl\{\int_{0}^\infty r^{N-1} \sup_{0<s\leq r}\int_{\sn}&\frac{|u(s\sigma)|^p}{r^p}\dsn\dr,  \int_{0}^\infty  r^{N-1}\sup_{r\leq s<\infty}\int_{\sn}\frac{|u(s\sigma)|^p}{s^p}\dsn\dr\biggr\}\nonumber\\&\leq  \bigg|\frac{p}{N-p}\bigg|^p\int_{\rn}\bigg|\frac{x}{|x|}\cdot \nabla u(x)\bigg|^p\dx,
		\end{align}
		with the sharp constant. More precisely, the Hardy constant $\big|\frac{p}{N-p}\big|^p$ is sharp in the sense that no inequality of the form 
		\begin{align*}
				\max \biggl\{\int_{0}^\infty r^{N-1} \sup_{0<s\leq r}\int_{\sn}&\frac{|u(s\sigma)|^p}{r^p}\dsn\dr,  \int_{0}^\infty  r^{N-1}\sup_{r\leq s<\infty}\int_{\sn}\frac{|u(s\sigma)|^p}{s^p}\dsn\dr\biggr\}\\&\leq  C\: \int_{\rn}\bigg|\frac{x}{|x|}\cdot \nabla u(x)\bigg|^p\dx,
		\end{align*}
		holds, for  all $ u \in C_{c}^{\infty}(\rn \setminus \{o\})$, when $C < \big|\frac{p}{N-p}\big|^p$.
	\end{theorem}
	\begin{proof}
		Let $u\in C_c^\infty(\rn \setminus \{o\})$ and $\tilde{u}$ be the radial symmetric function  associated to it. Then exploiting Lemma \ref{sup_lem} with $f(x)=1$ and then substituting the result into Theorem \ref{main_th_rad}, we deduce
		\begin{align*}
			&	\max \biggl\{\int_{0}^\infty r^{N-1} \sup_{0<s\leq r}\int_{\sn}\frac{|u(s\sigma)|^p}{r^p}\dsn\dr,  \int_{0}^\infty  r^{N-1}\sup_{r\leq s<\infty}\int_{\sn}\frac{|u(s\sigma)|^p}{s^p}\dsn\dr\biggr\}\\&\leq  \int_{\rn} \max\biggl\{ \sup_{\bar{B}(o\:;\:|x|)\setminus\{o\}}\frac{|\tilde{u}(y)|^p}{|x|^p}\:,\: \sup_{B^c(o\:;\:|x|)}\frac{|\tilde{u}(y)|^p}{|y|^p}\biggr\}\dx\\&\leq \bigg|\frac{p}{N-p}\bigg|^p\int_{\rn}\bigg|\frac{x}{|x|}\cdot \nabla \tilde{u}(x)\bigg|^p\dx.
		\end{align*}
		
		Next, using Lemma \ref{tran_lem} with $f(x)=1$, we have
		\begin{align*}
			\int_{\rn}\bigg|\frac{x}{|x|}\cdot \nabla \tilde{u}(x)\bigg|^p\dx\leq \int_{\rn}\bigg|\frac{x}{|x|}\cdot \nabla u(x)\bigg|^p\dx.
		\end{align*}
		
		Finally, combining the above two estimates we obtain
		\begin{align*}
				\max \biggl\{\int_{0}^\infty r^{N-1} \sup_{0<s\leq r}\int_{\sn}&\frac{|u(s\sigma)|^p}{r^p}\dsn\dr,  \int_{0}^\infty  r^{N-1}\sup_{r\leq s<\infty}\int_{\sn}\frac{|u(s\sigma)|^p}{s^p}\dsn\dr\biggr\}\\&\leq  \bigg|\frac{p}{N-p}\bigg|^p\int_{\rn}\bigg|\frac{x}{|x|}\cdot \nabla u(x)\bigg|^p\dx,
		\end{align*}
		which is the desired result \eqref{eqn_main_th}. The sharpness follows from the known sharp constant in the classical setup.
	\end{proof}

	\begin{remark}\label{rem-5}
		On the left-hand side of \eqref{eqn_main_th}, we have 
		\begin{align*}
			\max \biggl\{\int_{0}^\infty r^{N-1} \sup_{0<s\leq r}\int_{\sn}&\frac{|u(s\sigma)|^p}{r^p}\dsn\dr,  \int_{0}^\infty  r^{N-1}\sup_{r\leq s<\infty}\int_{\sn}\frac{|u(s\sigma)|^p}{s^p}\dsn\dr\biggr\}\\&\geq\int_{0}^{\infty}\int_{\sn}r^{N-1}\:\frac{|u(r\sigma)|^p}{r^p}\dsn\dr\\&=\int_{\rn}\frac{|u(x)|^p}{|x|^p}\dx.
		\end{align*}
		By using Gauss's lemma on the right-hand side of \eqref{eqn_main_th}, we get
		$$\bigg|\frac{x}{|x|}\cdot \nabla u(x)\bigg|\leq |\nabla u(x)|.$$ Combining the above two facts, we conclude that \eqref{eqn_main_th} is an improvement of the sharp Hardy inequality.
	\end{remark}

	\section{Uncertainty principle}\label{5}
	In this section, we focus on the Heisenberg-Pauli-Weyl (HPW) type uncertainty principle, which can be obtained immediately from the obtained new version of the Hardy inequality. HPW uncertainty principle has several physical and mathematical applications. In physics, uncertainty principles may be used for establishing the stability of matter. In quantum mechanics, the uncertainty principle implies that both the momentum and the position of an object cannot be exactly measured at the same time. The most well-known mathematical formulation of the uncertainty principle is probably the HPW uncertainty principle. First, we present the result for radial functions and then for the non-radial setting.
	
	\begin{theorem}
		Let $1\leq N<p<\infty$. Then, for any $u\in C_{c,rad}^\infty(\rn\setminus \{o\})$, the following uncertainty principle holds:
		\begin{align*}
		&\max\biggl\{\int_{\g}\sup_{\bar{B}(o\:;\:|x|)\setminus\{o\}}|u(y)|^p\dx,\int_{\g}\sup_{B^c(o\:;\:|x|)}|u(y)|^p\dx\biggr\}	\nonumber\\& \leq \bigg|\frac{p}{N-p}\bigg|\max\biggl\{ \bigg(\int_{\g}\sup_{\bar{B}(o\:;\:|x|)\setminus\{o\}}|x|^\frac{p}{p-1}|u(y)|^{p}\dx\bigg)^{\frac{p-1}{p}},\\&\quad\quad\quad\quad\quad\quad\quad\quad\bigg(\int_{\g}\sup_{B^c(o\:;\:|x|)}|y|^\frac{p}{p-1}|u(y)|^{p}\dx\bigg)^{\frac{p-1}{p}}\biggr\}\bigg(\int_{\g}\bigg|\frac{x}{|x|}\cdot \nabla  u(x)\bigg|^p\dx\bigg)^{\frac{1}{p}}.
	\end{align*}		
	\end{theorem}
	\begin{proof}
		For $u\in C_{c,rad}^\infty(\g\setminus\{o\})$ we estimate each term separately. Let us begin with
		\begin{align*}
			&\int_{\g}\sup_{\bar{B}(o\:;\:|x|)\setminus\{o\}}|u(y)|^p\dx\\
			&\leq\int_{\g}\bigg(\sup_{\bar{B}(o\:;\:|x|)\setminus\{o\}}|x||u(y)|^{p-1}\bigg)\bigg(\sup_{\bar{B}(o\:;\:|x|)\setminus\{o\}}\frac{|u(y)|}{|x|}\bigg)\dx\\
			&\leq\bigg(\int_{\g}\bigg(\sup_{\bar{B}(o\:;\:|x|)\setminus\{o\}}|x||u(y)|^{p-1}\bigg)^{\frac{p}{p-1}}\dx\bigg)^{\frac{p-1}{p}} \bigg(\int_{\g}\bigg(\sup_{\bar{B}(o\:;\:|x|)\setminus\{o\}}\frac{|u(y)|}{|x|}\bigg)^p\dx\bigg)^{\frac{1}{p}}\\
			&=\bigg(\int_{\g}\sup_{\bar{B}(o\:;\:|x|)\setminus\{o\}}|x|^\frac{p}{p-1}|u(y)|^{p}\dx\bigg)^{\frac{p-1}{p}} \bigg(\int_{\g}\sup_{\bar{B}(o\:;\:|x|)\setminus\{o\}}\frac{|u(y)|^p}{|x|^p}\dx\bigg)^{\frac{1}{p}}\\
			&=\bigg(\int_{\g}\sup_{\bar{B}(o\:;\:|x|)\setminus\{o\}}|x|^\frac{p}{p-1}|u(y)|^{p}\dx\bigg)^{\frac{p-1}{p}}\times\\& \quad\quad\quad\quad\quad\quad\quad\quad\bigg(\int_{\rn} \max\biggl\{ \sup_{\bar{B}(o\:;\:|x|)\setminus\{o\}}\frac{|u(y)|^p}{|x|^p}\:,\:\sup_{B^c(o\:;\:|x|)}\frac{|u(y)|^p}{|y|^p}\biggr\}\dx\bigg)^{\frac{1}{p}}\\
			&\leq \bigg|\frac{p}{N-p}\bigg| \bigg(\int_{\g}\sup_{\bar{B}(o\:;\:|x|)\setminus\{o\}}|x|^\frac{p}{p-1}|u(y)|^{p}\dx\bigg)^{\frac{p-1}{p}}\bigg(\int_{\g}\bigg|\frac{x}{|x|}\cdot \nabla  u(x)\bigg|^p\dx\bigg)^{\frac{1}{p}}.
		\end{align*}
		In the above we have used the H\"older inequality, Lemma \ref{help_lem_3} twice, and finally Theorem \ref{main_th_rad} step by step. Exploiting the similar steps as earlier, we compute
		\begin{align*}
			&\int_{\g}\sup_{B^c(o\:;\:|x|)}|u(y)|^p\dx\\
			&\leq \int_{\g}\bigg(\sup_{B^c(o\:;\:|x|)}|y||u(y)|^{p-1}\bigg)\bigg(\sup_{B^c(o\:;\:|x|)}\frac{|u(y)|}{|y|}\bigg)\dx\\
			&\leq \bigg(\int_{\g}\bigg(\sup_{B^c(o\:;\:|x|)}|y||u(y)|^{p-1}\bigg)^{\frac{p}{p-1}}\dx\bigg)^{\frac{p-1}{p}} \bigg(\int_{\g}\bigg(\sup_{B^c(o\:;\:|x|)}\frac{|u(y)|}{|y|}\bigg)^p\dx\bigg)^{\frac{1}{p}}\\
			&=\bigg(\int_{\g}\sup_{B^c(o\:;\:|x|)}|y|^\frac{p}{p-1}|u(y)|^{p}\dx\bigg)^{\frac{p-1}{p}} \bigg(\int_{\g}\sup_{B^c(o\:;\:|x|)}\frac{|u(y)|^p}{|y|^p}\dx\bigg)^{\frac{1}{p}}\\
			&=\bigg(\int_{\g}\sup_{B^c(o\:;\:|x|)}|y|^\frac{p}{p-1}|u(y)|^{p}\dx\bigg)^{\frac{p-1}{p}}\times\\& \quad\quad\quad\quad\quad\quad\quad\quad\bigg(\int_{\rn} \max\biggl\{ \sup_{\bar{B}(o\:;\:|x|)\setminus\{o\}}\frac{|u(y)|^p}{|x|^p}\:,\:\sup_{B^c(o\:;\:|x|)}\frac{|u(y)|^p}{|y|^p}\biggr\}\dx\bigg)^{\frac{1}{p}}\\
			&\leq \bigg|\frac{p}{N-p}\bigg| \bigg(\int_{\g}\sup_{B^c(o\:;\:|x|)}|y|^\frac{p}{p-1}|u(y)|^{p}\dx\bigg)^{\frac{p-1}{p}}\bigg(\int_{\g}\bigg|\frac{x}{|x|}\cdot \nabla  u(x)\bigg|^p\dx\bigg)^{\frac{1}{p}}.
		\end{align*}
		Finally, combining both cases, we arrive at the desired result. 
	\end{proof}

Now we state the version for non-radial functions.
	\begin{theorem}
	Let $1\leq N<p<\infty$. Then for any $u\in C_{c}^\infty(\rn\setminus \{o\})$, the following uncertainty principle holds:
	\begin{align*}
		&\max \biggl\{\int_{0}^\infty r^{N-1} \sup_{0<s\leq r}\int_{\sn}\frac{|u(s\sigma)|^p}{r^p}\dsn\dr,  \int_{0}^\infty  r^{N-1}\sup_{r\leq s<\infty}\int_{\sn}\frac{|u(s\sigma)|^p}{s^p}\dsn\dr\biggr\}\\& \leq \bigg|\frac{p}{N-p}\bigg|\max\biggl\{ \bigg(\int_{0}^\infty r^{N-1}\sup_{0<s\leq r}\int_{\sn}r^{\frac{p}{p-1}}|u(s\sigma)|^p\dsn\dr\bigg)^{\frac{p-1}{p}},\\&\quad\quad\bigg(\int_{0}^\infty r^{N-1}\sup_{r\leq s <\infty}\int_{\sn}s^{\frac{p}{p-1}}|u(s\sigma)|^p\dsn\dr\bigg)^{\frac{p-1}{p}}\biggr\}\bigg(\int_{\g}\bigg|\frac{x}{|x|}\cdot \nabla  u(x)\bigg|^p\dx\bigg)^{\frac{1}{p}}.
	\end{align*}		
\end{theorem}
\begin{proof}
	Let $u\in C_c^\infty(\g\setminus\{o\})$ and using polar coordinates we have
\begin{align*}
	&\int_{0}^\infty r^{N-1} \sup_{0<s\leq r}\int_{\sn}|u(s\sigma)|^p\dsn\dr\\
	&=\int_{0}^\infty r^{N-1} \sup_{0<s\leq r}\int_{\sn}\frac{|u(s\sigma)|}{r}\: \cdot r \: |u(s\sigma)|^{p-1}\dsn\dr\\
	&\leq \int_{0}^\infty r^{N-1}\sup_{0<s\leq r}\bigg(\int_{\sn}\frac{|u(s\sigma)|^p}{r^p}\dsn\bigg)^{\frac{1}{p}}\bigg(\int_{\sn}r^{\frac{p}{p-1}}|u(s\sigma)|^p\dsn\bigg)^{\frac{p-1}{p}}\dr\\
	&\leq \int_{0}^\infty r^{N-1}\sup_{0<s\leq r}\bigg(\int_{\sn}\frac{|u(s\sigma)|^p}{r^p}\dsn\bigg)^{\frac{1}{p}}\sup_{0<s\leq r}\bigg(\int_{\sn}r^{\frac{p}{p-1}}|u(s\sigma)|^p\dsn\bigg)^{\frac{p-1}{p}}\dr\\
	&= \int_{0}^\infty \biggl\{r^{\frac{N-1}{p}}\sup_{0<s\leq r}\bigg(\int_{\sn}\frac{|u(s\sigma)|^p}{r^p}\dsn\bigg)^{\frac{1}{p}}\biggr\}\times\\&\quad\quad\quad\quad\biggl\{r^{\frac{(N-1)(p-1)}{p}}\sup_{0<s\leq r}\bigg(\int_{\sn}r^{\frac{p}{p-1}}|u(s\sigma)|^p\dsn\bigg)^{\frac{p-1}{p}}\biggr\}\dr\\
	&\leq \bigg(\int_{0}^\infty r^{N-1}\sup_{0<s\leq r}\int_{\sn}\frac{|u(s\sigma)|^p}{r^p}\dsn\dr\bigg)^{\frac{1}{p}}\times\\
	&\quad\quad\quad\quad\bigg(\int_{0}^\infty r^{N-1}\sup_{0<s\leq r}\int_{\sn}r^{\frac{p}{p-1}}|u(s\sigma)|^p\dsn\dr\bigg)^{\frac{p-1}{p}}\\
	&\leq\bigg(\max \biggl\{\int_{0}^\infty r^{N-1} \sup_{0<s\leq r}\int_{\sn}\frac{|u(s\sigma)|^p}{r^p}\dsn\dr, \\
	&\quad\quad\quad\quad\int_{0}^\infty  r^{N-1}\sup_{r\leq s<\infty}\int_{\sn}\frac{|u(s\sigma)|^p}{s^p}\dsn\dr\biggr\}\bigg)^{\frac{1}{p}}\times\\&\quad\quad\quad\quad\quad\quad\quad\quad\bigg(\int_{0}^\infty r^{N-1}\sup_{0<s\leq r}\int_{\sn}r^{\frac{p}{p-1}}|u(s\sigma)|^p\dsn\dr\bigg)^{\frac{p-1}{p}}\\
	&\leq \bigg|\frac{p}{N-p}\bigg| \bigg(\int_{\g}\bigg|\frac{x}{|x|}\cdot \nabla  u(x)\bigg|^p\dx\bigg)^{\frac{1}{p}}\bigg(\int_{0}^\infty r^{N-1}\sup_{0<s\leq r}\int_{\sn}r^{\frac{p}{p-1}}|u(s\sigma)|^p\dsn\dr\bigg)^{\frac{p-1}{p}}.
\end{align*}
	In the above we have exploited  the H\"older inequality, Lemma \ref{help_lem_3}, and finally Theorem \ref{main_th} step by step. Following the similar steps as earlier, we compute
	\begin{align*}
	&\int_{0}^\infty r^{N-1} \sup_{r\leq s<\infty}\int_{\sn}|u(s\sigma)|^p\dsn\dr\\
	&=\int_{0}^\infty r^{N-1} \sup_{r\leq s<\infty}\int_{\sn}\frac{|u(s\sigma)|}{s}\cdot s\:|u(s\sigma)|^{p-1}\dsn\dr\\
	&\leq \int_{0}^\infty r^{N-1}\sup_{r\leq s<\infty}\bigg(\int_{\sn}\frac{|u(s\sigma)|^p}{s^p}\dsn\bigg)^{\frac{1}{p}}\bigg(\int_{\sn}s^{\frac{p}{p-1}}|u(s\sigma)|^p\dsn\bigg)^{\frac{p-1}{p}}\dr\\
	&\leq \int_{0}^\infty r^{N-1}\sup_{r\leq s<\infty}\bigg(\int_{\sn}\frac{|u(s\sigma)|^p}{s^p}\dsn\bigg)^{\frac{1}{p}}\sup_{r\leq s<\infty}\bigg(\int_{\sn}s^{\frac{p}{p-1}}|u(s\sigma)|^p\dsn\bigg)^{\frac{p-1}{p}}\dr\\
	&=\int_{0}^\infty \biggl\{r^{\frac{N-1}{p}}\sup_{r\leq s<\infty}\bigg(\int_{\sn}\frac{|u(s\sigma)|^p}{s^p}\dsn\bigg)^{\frac{1}{p}}\biggr\}\times\\
	&\quad\quad\quad\quad\biggl\{r^{\frac{(N-1)(p-1)}{p}}\sup_{ r\leq s<\infty}\bigg(\int_{\sn}s^{\frac{p}{p-1}}|u(s\sigma)|^p\dsn\bigg)^{\frac{p-1}{p}}\biggr\}\dr\\
	&\leq \bigg(\int_{0}^\infty r^{N-1}\sup_{r\leq s<\infty}\int_{\sn}\frac{|u(s\sigma)|^p}{s^p}\dsn\dr\bigg)^{\frac{1}{p}}\times\\
	&\quad\quad\quad\quad\bigg(\int_{0}^\infty r^{N-1}\sup_{ r\leq s<\infty}\int_{\sn}s^{\frac{p}{p-1}}|u(s\sigma)|^p\dsn\dr\bigg)^{\frac{p-1}{p}}\\
	&\leq\bigg(\max \biggl\{\int_{0}^\infty r^{N-1} \sup_{0<s\leq r}\int_{\sn}\frac{|u(s\sigma)|^p}{r^p}\dsn\dr, \\& \quad\quad\quad\quad\int_{0}^\infty  r^{N-1}\sup_{r\leq s<\infty}\int_{\sn}\frac{|u(s\sigma)|^p}{s^p}\dsn\dr\biggr\}\bigg)^{\frac{1}{p}}\times\\
	&\quad\quad\quad\quad\quad\quad\bigg(\int_{0}^\infty r^{N-1}\sup_{ r\leq s<\infty}\int_{\sn}s^{\frac{p}{p-1}}|u(s\sigma)|^p\dsn\dr\bigg)^{\frac{p-1}{p}}\\
	&\leq \bigg|\frac{p}{N-p}\bigg| \bigg(\int_{\g}\bigg|\frac{x}{|x|}\cdot \nabla  u(x)\bigg|^p{\rm d}x\bigg)^{\frac{1}{p}}\bigg(\int_{0}^\infty r^{N-1}\sup_{r\leq  s <\infty}\int_{\sn}s^{\frac{p}{p-1}}|u(s\sigma)|^p\dsn\dr\bigg)^{\frac{p-1}{p}}.
\end{align*}
	By combining both cases, we immediately arrive at the desired result. 
\end{proof}

	\section*{Acknowledgements} 
	\noindent 
	The authors would like to thank Rupert Frank for pointing out a problem in the first version of this paper.  This project was discussed when the authors met at the Ghent Analysis \& PDE center at Ghent University in Summer 2022. P. R. and D. S.  would like to thank the university for their support and hospitality. The authors were supported by the FWO Odysseus 1 grant G.0H94.18N: Analysis and Partial Differential Equations and by the Methusalem programme of the Ghent University Special Research Fund (BOF) (Grant number 01M01021). M. R. was also supported by EPSRC grant EP/R003025/2.
	
	\medspace

\end{document}